\documentclass{article}
\usepackage{arxiv}
\pdfoutput=1

\usepackage{xcolor} 	
\usepackage{lipsum} 	
\usepackage{cite}   	
\usepackage{amsfonts}	
\usepackage{amsmath}	
\usepackage{amssymb}    
\usepackage{alphabeta}	
\usepackage{url}        
\usepackage{graphicx}   
\usepackage{nameref}	
\usepackage{titlesec}	
\usepackage{algorithm}
\usepackage{algpseudocodex}
\usepackage{float}   
\usepackage{subfig}
\usepackage{caption}
\usepackage{titlesec}
\usepackage{abstract}
\usepackage[shortlabels, inline]{enumitem}
\usepackage{setspace}
\usepackage{colortbl}
\usepackage{cases}

\newtheorem{theorem}{Theorem}

\newtheorem{definition}{Definition}
\newtheorem{assumption}{Assumption}

\newtheorem{remark}{Remark}
\newenvironment{proof}{\paragraph*{Proof}}{\hfill$\square$}

\newcommand{\diag}[1]{\operatorname{diag}\left\{#1\right\}}
\newcommand{\norm}[2]{ \left\Vert {#1} \right\Vert_{#2}}
\newcommand{\normsmall}[2]{\Vert {#1} \Vert_{#2}}

\newcommand{\authornote}[2]{#1 et al. \cite{#2}} 
\newcommand{\settablestretch}{\setstretch{1.1}}
\newcommand{\settablefontsize}{\normalsize}

\newcommand{\colorrowoftable}{\rowcolor{gray!40}}

\newcommand{\proofreference}{Appendix A}

\newcommand{\RNN}{RNN}
\newcommand{\FFNN}{FFNN}
\newcommand{\LSTM}{LSTM}
\newcommand{\GRU}{GRU}
\newcommand{\FC}{FC}

\newcommand{\MPC}{MPC}
\newcommand{\MIMO}{MIMO}
\newcommand{\ISS}{ISS}
\newcommand{\ISSp}[1]{ISS$_{#1}$}
\newcommand{\strictISS}{\ISSp{\infty}}
\newcommand{\deltaISS}{$\delta$ISS}
\newcommand{\deltaISSp}[1]{$\delta$ISS$_{#1}$}
\newcommand{\strictdeltaISS}{\deltaISSp{\infty}}

\newcommand{\HR}{HR} 
\newcommand{\SSR}{SSR} 
\newcommand{\EMR}{EMR} 
\newcommand{\PWM}{PWM} 

\newcommand{\MSE}{MSE}

\newcommand{\margin}{\gamma}
\newcommand{\ISSpenalty}{\rho}
\newcommand{\learnrate}{\eta}

\newcommand{\figurenametext}{\figurename{}} 
\newcommand{\paragraphfont}[1]{\textbf{{#1}.}}

\titlespacing*{\section}{0pt}{3pt}{3pt}
\setlist[itemize]{left=0pt}

\allowdisplaybreaks

\begin{document}

\onecolumn

\title{Infinity-norm-based Input-to-State-Stable Long Short-Term Memory networks: a thermal systems perspective}

\author{%
	\begin{tabular}[t]{@{}c@{}}
		Stefano De Carli\thanks{S. De Carli, D. Previtali, L. Pitturelli, M. Mazzoleni, A. Ferramosca, and F. Previdi are with the Department of\\Management, Information and Production Engineering, University of Bergamo, Via G. Marconi 5, 24044 Dalmine (BG), Italy.\\{\tt\small stefano.decarli@unibg.it}} \\
		Mirko Mazzoleni\footnotemark[1]
	\end{tabular}%
	\hspace{1cm}%
	\begin{tabular}[t]{@{}c@{}}
		Davide Previtali\footnotemark[1] \\
		Antonio Ferramosca\footnotemark[1]
	\end{tabular}%
	\hspace{1cm}%
	\begin{tabular}[t]{@{}c@{}}
		Leandro Pitturelli\footnotemark[1] \\
		Fabio Previdi\footnotemark[1]
	\end{tabular}%
}

\date{}

\twocolumn[
    \maketitle
    \vspace{-25pt}

    \begin{abstract}
        Recurrent Neural Networks (\RNN{}s) have shown remarkable performances in system identification, particularly in nonlinear dynamical systems such as thermal processes. However, stability remains a critical challenge in practical applications: although the underlying process may be intrinsically stable, there may be no guarantee that the resulting \RNN{} model captures this behavior. This paper addresses the stability issue by deriving a sufficient condition for Input-to-State Stability based on the infinity-norm (\strictISS{}) for Long Short-Term Memory (\LSTM{}) networks. The obtained condition depends on fewer network parameters compared to prior works. 
        A \strictISS{}-promoted training strategy is developed, incorporating a penalty term in the loss function that encourages stability and 
        an ad hoc early stopping approach.
        The quality of \LSTM{} models trained via the proposed approach is validated on a thermal system case study, where the \strictISS{}-promoted \LSTM{} outperforms both a physics-based 
        model and an \strictISS{}-promoted Gated Recurrent Unit (\GRU{}) network while also surpassing non-\strictISS{}-promoted \LSTM{} and \GRU{} \RNN{}s.
    \end{abstract}
    \vspace{\baselineskip}
]
\saythanks

\section{Introduction}
\label{s:introduction}
Recurrent Neural Networks (\RNN s) have become widely adopted in the control systems and identification community due to their ability to model dynamical behavior through internal memory mechanisms \cite{schafer_recurrent_2007}. Unlike Feed-Forward Neural Networks (\FFNN{}s), which struggle to capture temporal dependencies, \RNN{} architectures such as Long Short-Term Memory (\LSTM{}) networks \cite{terzi_learning_2021} and Gated Recurrent Unit (\GRU{}) networks \cite{bonassi_stability_2021} effectively address this challenge \cite{ljung_deep_2020}. 
The popularity of \LSTM{}s and \GRU{}s is largely due to their superior performance in system identification tasks, where deep learning architectures can successfully capture complex nonlinearities and can intrinsically handle Multiple-Input Multiple-Output (\MIMO{}) systems \cite{ljung_deep_2020}. Furthermore, these networks can be seamlessly integrated into traditional model-based control strategies, such as Model Predictive Control (\MPC{}), due to their state-space formulation \cite{terzi_learning_2021}.

Thermal systems, characterized by their inherent nonlinearities, dynamical behavior, and \MIMO{} nature, are well-suited for \RNN{} identification. As a matter of fact, \RNN{}s have been successfully applied for modelling the cooling systems of buildings \cite{terzi_learning-based_2020}, heat exchangers \cite{vasickaninova_neural_2011, mohanraj_applications_2015}, photovoltaic-thermal systems \cite{al-waeli_artificial_2019}, and refrigeration and air-conditioning systems \cite{mohanraj_applications_2012}, exhibiting improved performance compared to traditional methods for system identification and control.

Stability is a critical property when deploying \RNN{}s in real-world applications, 
especially when dealing with systems that intrinsically exhibit a stable behavior, such as thermal systems. In this context, the temperature change due to a limited perturbation in the amount of heat supplied to a thermal system is typically bounded. 
Consequently, it can be beneficial to have a model that satisfies certain stability properties, consistently with the plant itself. Given that recurrent neural networks are nonlinear models with external inputs, we rely on the notion of Input-to-State Stability (\ISS{}) \cite{jiang_input--state_2001}, which guarantees that, for any initial condition and any bounded input, the system's states remain bounded over time. Additionally, Incremental Input-to-State Stability (\deltaISS{}) extends this concept by guaranteeing that the difference between two states trajectories asymptotically decreases as the input differences diminish. 
In practice, a system is \ISS{} if we can derive an upper bound on the norm of its state vector that is based on the initial states and the input sequence supplied to it (and similarly for \deltaISS{}). In this work, for the sake of clarity, we explicit the $p$-norm used in the definition of input-to-state stability using the notation \ISSp{p} (respectively, \deltaISSp{p}). The most commonly used stability notions are \ISSp{\infty}/\deltaISSp{\infty} and \ISSp{2}/\deltaISSp{2}, which capture the maximum and average deviation over time respectively \cite{sontag_input--state_1995}. For what concerns \LSTM{} networks, \authornote{Terzi}{terzi_learning_2021} established sufficient parametric conditions for both \ISSp{2} and \deltaISSp{2}. Instead, \authornote{Bonassi}{bonassi_stability_2021} analyzed the \strictISS{} and \strictdeltaISS{} of \GRU{} networks. However, at the moment, there are no works in the literature that tackle directly the input-to-state stability property based on the $\infty$-norm for \LSTM{} networks, which is the purpose of this work.

\paragraphfont{Contributions}
The main contributions in the present paper are: (i) we derive a sufficient condition for \strictISS{} of \LSTM{} networks which, compared to \cite[Proposition 1]{terzi_learning_2021}, depends on fewer network parameters, (ii) we develop a training strategy that promotes the \strictISS{} property, 
and (iii) we employ the \LSTM{} trained according to the proposed approach on a thermal system identification case study, 
comparing it to a grey-box physics-based model and the \strictISS{}-promoted \GRU{} in \cite{bonassi_stability_2021}.
The designed training strategy is also compared to standard \LSTM{}/\GRU{} training (for which \strictISS{} is not promoted in any way) on the same case study.

\paragraphfont{Organization}
This paper is organized as follows. Section \ref{s:problem_statement} introduces the problem statement. Then, Section \ref{s:input_to_state_stable_lstm_networks} presents the \LSTM{} network under study, derives the sufficient parametric condition for \strictISS{}, and describes the proposed \strictISS{}-promoting training strategy. Afterwards, Section \ref{s:thermal_system_under_study} presents the thermal system case study and its respective physics-based model. Next, Section \ref{s:experimental_results} compares the performances of the network models and the physics-based model. Lastly, Section \ref{s:conclusion} gives some final remarks.

\paragraphfont{Notations}
We denote by $\mathbb{R}$ and $\mathbb{N}$ the set of real and natural numbers respectively ($0 \in \mathbb{N}$). Furthermore, $\mathbb{R}_{>0}$ and $\mathbb{R}_{\geq0}$ stand for the set of positive and non-negative real numbers respectively.
Given $n, m \in \mathbb{N}$, $\mathbb{R}^n$ is the set of real column vectors of dimension $n$, while $\mathbb{R}^{n \times m}$ is the set of real matrices of dimension $n \times m$.
Moreover, $\boldsymbol{1}_{n} \in \mathbb{R}^{n}$ is the $n$-dimensional column vector of ones, 
and $\diag{a_1, \ldots, a_n} \in \mathbb{R}^{n \times n}$ is the diagonal matrix with entries $a_1, \ldots, a_n \in \mathbb{R}$ on the main diagonal.
Next, $\circ$ denotes the Hadamard (element-wise) product, while $\norm{\cdot}{p}$ is the $p$-norm of either a matrix or a vector.
Further, $\left|\mathcal{S}\right|$ is the cardinality of the set $\mathcal{S}$.

We denote continuous-time signals  $s : \mathbb{R}_{\geq0} \to \mathbb{R}$ as $s(t)$, where $t \in \mathbb{R}_{\geq0}$ (in $\mathrm{s}$) is the time. Furthermore, $\dot{s}(t)$ is the derivative of $s$ w.r.t. $t$. Instead, $k \in \mathbb{N}$ is the discrete-time step, and $s_k$ is the discrete-time signal resulting from sampling $s(t)$ at a sampling time $T_{\mathrm{s}} \in \mathbb{R}_{> 0}$ (in $\mathrm{s}$), i.e. $s_k = s(k T_{\mathrm{s}}), \forall k \in \mathbb{N}$. The same notations are used for vectors of signals, which are written in a bold font.
Let $\boldsymbol{x}_k = \begin{bmatrix}
	x_{1, k} & \ldots & x_{n, k}
\end{bmatrix}^\top \in \mathbb{R}^n$ be a discrete-time signal with $n \in \mathbb{N}$ components $x_{\iota, k} \in \mathbb{R}, \iota \in \{1, \ldots, n\}$. Then, $\boldsymbol{y}_k = \boldsymbol{\sigma}(\boldsymbol{x}_k)$ and $\boldsymbol{z}_k = \boldsymbol{\tanh}(\boldsymbol{x}_k)$ denote the component-wise application of the sigmoid and hyperbolic tangent functions to $\boldsymbol{x}_k$.
\section{Problem statement}
\label{s:problem_statement}
We consider the problem of modelling and identifying a discrete-time \MIMO{} dynamical system with inputs $\boldsymbol{u}_k \in \mathbb{R}^{n_u}, n_u \in \mathbb{N}$, and outputs $\boldsymbol{y}_k \in \mathbb{R}^{n_y}, n_y \in \mathbb{N}$. Specifically, we are interested in nonlinear state-space models with state vector $\boldsymbol{x}_k \in \mathbb{R}^{n_x}, n_x \in \mathbb{N},$ that are input-to-state stable according to the following Definition.
\begin{definition}[Input-to-state stability (\strictISS{}) \cite{bonassi_stability_2021}]
	\label{def:strict_iss}
	A dynamical system with states $\boldsymbol{x}_k \in \mathcal{X} \subseteq \mathbb{R}^{n_x}$ and inputs $\boldsymbol{u}_k \in \mathcal{U} \subseteq \mathbb{R}^{n_u}$ is said to be input-to-state stable if there exist functions $\beta \in \mathcal{KL}$ and $\gamma_u, \gamma_b \in \mathcal{K}_\infty$ such that, for any $k \in \mathbb{N}$, any initial state $\boldsymbol{x}_0 \in \mathcal{X}$, any input sequence
	$\mathbf{u} = \{\boldsymbol{u}_h \in \mathcal{U}: h \in \{0, \ldots, k-1\}\}$,
	and any bias $\boldsymbol{b} \in \mathbb{R}^{n_x}$, it holds that:
	\begin{equation}
		\label{eq:strict_iss}
		\normsmall{\boldsymbol{x}_k}{\infty} \leq \beta(\normsmall{\boldsymbol{x}_0}{\infty}, k) + \gamma_u \left(\max_{0 \leq h < k} \normsmall{\boldsymbol{u}_h}{\infty} \right) + \gamma_b \left( \normsmall{\boldsymbol{b}}{\infty} \right).
	\end{equation}
\end{definition}
See \cite{jiang_input--state_2001,terzi_learning_2021} for details on $\mathcal{K}_\infty$ and $\mathcal{KL}$ functions.
%
\section{Input-to-state stable \LSTM{} networks}
\label{s:input_to_state_stable_lstm_networks}
An \LSTM{} network \cite{terzi_learning_2021} consists of $L \in \mathbb{N}$ layers, each with $n_{\mathrm{hu}}^{(l)} \in \mathbb{N}, l \in \{1, \ldots, L\},$ hidden units. 
The $l$-th layer of an \LSTM{} amounts to a discrete-time nonlinear dynamical system in state-space form whose state vector $\boldsymbol{x}_k^{(l)} \in \mathbb{R}^{2 n_{\mathrm{hu}}^{(l)}}$ is divided into two components: the cell state $\boldsymbol{c}_k^{(l)} \in \mathbb{R}^{n_{\mathrm{hu}}^{(l)}}$ and the hidden state $\boldsymbol{h}_k^{(l)} \in \mathbb{R}^{n_{\mathrm{hu}}^{(l)}}$, i.e. $\boldsymbol{x}_k^{(l)} = \begin{bmatrix} \boldsymbol{c}_k^{(l)^\top} & \boldsymbol{h}_k^{(l)^\top} \end{bmatrix}^\top$. Several gating mechanisms control the flow of information through the network. Specifically, the forget gate $\boldsymbol{f}_k^{(l)} \in \mathbb{R}^{n_\mathrm{hu}^{(l)}}$, input gate $\boldsymbol{i}_k^{(l)} \in \mathbb{R}^{n_\mathrm{hu}^{(l)}}$, output gate $\boldsymbol{o}_k^{(l)} \in \mathbb{R}^{n_\mathrm{hu}^{(l)}}$, and candidate memory $\boldsymbol{g}_k^{(l)} \in \mathbb{R}^{n_\mathrm{hu}^{(l)}}$ in layer $l$ are computed as:
\begin{subequations}
	\label{eq:gates_LSTM}
	\begin{align}
		\boldsymbol{f}_k^{(l)} &= \boldsymbol{\sigma}\left( W_f^{(l)} \tilde{\boldsymbol{u}}_k^{(l)} + R_f^{(l)} \boldsymbol{h}_k^{(l)} + \boldsymbol{b}_f^{(l)} \right), \\
		\boldsymbol{i}_k^{(l)} &= \boldsymbol{\sigma}\left( W_i^{(l)} \tilde{\boldsymbol{u}}_k^{(l)} + R_i^{(l)} \boldsymbol{h}_k^{(l)} + \boldsymbol{b}_i^{(l)} \right), \\
		\boldsymbol{o}_k^{(l)} &= \boldsymbol{\sigma}\left( W_o^{(l)} \tilde{\boldsymbol{u}}_k^{(l)} + R_o^{(l)} \boldsymbol{h}_k^{(l)} + \boldsymbol{b}_o^{(l)} \right), \\
		\boldsymbol{g}_k^{(l)} &= \boldsymbol{\tanh}\left( W_g^{(l)} \tilde{\boldsymbol{u}}_k^{(l)} + R_g^{(l)} \boldsymbol{h}_k^{(l)} + \boldsymbol{b}_g^{(l)} \right),\\
		\label{eq:input_for_each_layer}
		\tilde{\boldsymbol{u}}_k^{(l)} &= \begin{cases}
			\boldsymbol{u}_k & \mathrm{if}\ l = 1,\\
			\boldsymbol{h}_{k+1}^{(l-1)} & \mathrm{if}\ l \in \{2, \ldots, L\},
		\end{cases}
	\end{align}
\end{subequations}
where $\tilde{\boldsymbol{u}}_k^{(l)}$ is the input for the $l$-th layer, $W_j^{(l)}$ are the input weights and in particular $W_j^{(1)} \in \mathbb{R}^{n_{\mathrm{hu}}^{(1)} \times n_u}$ while $W_j^{(l)} \in \mathbb{R}^{n_{\mathrm{hu}}^{(l)} \times n_{\mathrm{hu}}^{(l-1)}}$ for $l \in \{2, \ldots, L\}$, $R_j^{(l)} \in \mathbb{R}^{n_{\mathrm{hu}}^{(l)} \times n_{\mathrm{hu}}^{(l)}}$ are the recurrent weights, and $\boldsymbol{b}_j^{(l)} \in \mathbb{R}^{n_{\mathrm{hu}}^{(l)}}$ are the biases for $j \in \{f, i, o, g\}$. 
An \LSTM{} network is built via the concatenation of $L$ \LSTM{} layers followed by a Fully Connected (\FC{}) layer that produces the model output. In particular, 
the nonlinear state-space model of the network amounts to \cite{terzi_learning_2021}:
\begin{subequations}
	\label{eq:LSTM_model}
	\begin{numcases}{}
		\label{eq:LSTM_layer_cell_state_update}
		\boldsymbol{c}_{k+1}^{(l)} = \boldsymbol{f}_k^{(l)} \circ \boldsymbol{c}_k^{(l)} + \boldsymbol{i}_k^{(l)} \circ \boldsymbol{g}_k^{(l)} & $\forall l \in \{1, \ldots, L\}$, \\
		\label{eq:LSTM_layer_hidden_state_update}
		\boldsymbol{h}_{k+1}^{(l)} = \boldsymbol{o}_k^{(l)} \circ \boldsymbol{\tanh}\left( \boldsymbol{c}_{k+1}^{(l)} \right) & $\forall l \in \{1, \ldots, L\}$, \\
		\label{eq:LSTM_layer_output_equation}
		\boldsymbol{y}_k = W_y \boldsymbol{h}_{k+1}^{(L)} + \boldsymbol{b}_y, &
	\end{numcases}
\end{subequations}
where \eqref{eq:LSTM_layer_cell_state_update}/\eqref{eq:LSTM_layer_hidden_state_update} are the state update equations for each layer, while $W_y \in \mathbb{R}^{n_y \times n_{\mathrm{hu}}^{(L)}}$ and $\boldsymbol{b}_y \in \mathbb{R}^{n_y}$ are the weight matrix and bias for the \FC{} layer. Overall, the model in \eqref{eq:LSTM_model} has $n_x = 2\sum_{l=1}^{L} n_{\mathrm{hu}}^{(l)}$ states, namely
\begin{equation*}
	\boldsymbol{x}_k = \begin{bmatrix}
		\boldsymbol{c}_{k}^{(1)^\top} & \boldsymbol{h}_{k}^{(1)^\top} & \cdots & \boldsymbol{c}_{k}^{(L)^\top} & \boldsymbol{h}_{k}^{(L)^\top}
	\end{bmatrix}^\top \in \mathbb{R}^{n_x},
\end{equation*}
and relies on a set of parameters
\begin{align}
	\label{eq:LSTM_parameters}
	\theta &= \left\{W_j^{(l)}, R_j^{(l)}, \boldsymbol{b}_j^{(l)}: j \in \{f, i, o, g\}, l \in \{1, \ldots, L\}\right\} \nonumber \\
	&\quad \cup \{W_y, \boldsymbol{b}_y\}
\end{align}
that is learnt during \LSTM{} training (see Section \ref{ss:iss_promoted_lstm_training}). We also remark that $L$ and $ n_{\mathrm{hu}}^{(l)}, l \in \{1, \ldots, L\},$ are hyper-parameters that need to be tuned for network training.

\subsection{Stability analysis}
\label{ss:stability_analysis}
To analyze the stability of an \LSTM{} network, we make the following Assumptions on its inputs and initial states.
\begin{assumption}[Input vector boundedness]
	\label{as:boundedness_of_input}
	The inputs $\boldsymbol{u}_k$ for the model in \eqref{eq:LSTM_model} satisfy:
	\begin{equation}
		\label{eq:input_bound}
		\boldsymbol{u}_k \in \mathcal{U} = \{\boldsymbol{u}: - \boldsymbol{u}_{\max} \leq \boldsymbol{u} \leq \boldsymbol{u}_{\max}\}, \quad \forall k \in \mathbb{N},
	\end{equation}
	where $\boldsymbol{u}_{\max} \in \mathbb{R}^{n_u}$ is an upper bound on $\boldsymbol{u}_k$. 
\end{assumption}
\begin{assumption}[Initial states for the \LSTM{}]
	\label{ass:initial_hidden_states}
	The hidden states 
	for the model in \eqref{eq:LSTM_model} are initialized as follows:
	\begin{equation*}
		\label{eq:initial_hidden_states}
		\boldsymbol{h}_0^{(l)} \in \left(-1, 1\right)^{n_{\mathrm{hu}}^{(l)}}, \quad \forall l \in \{1, \ldots, L\}.
	\end{equation*}
	Instead, we make no assumption on the initialization of the cell states, 
	i.e. $\boldsymbol{c}_0^{(l)} \in \mathbb{R}^{n_{\mathrm{hu}}^{(l)}}$, $\forall l \in \{1, \ldots, L\}$.
\end{assumption}
The following Theorem provides a sufficient condition on the parameters of an \LSTM{} layer that ensures its \strictISS{} as in Definition \ref{def:strict_iss}.
\begin{theorem}[\strictISS{} for an \LSTM{} layer]
	\label{th:lstm_iss}
	The $l$-th \LSTM{} layer in \eqref{eq:LSTM_layer_cell_state_update}/\eqref{eq:LSTM_layer_hidden_state_update}, $l \in \{1, \ldots, L\}$, is \strictISS{} if the following sufficient condition holds:
	\begin{equation}
		\label{eq:lstm_stability_condition}
		\bar{\sigma}_f^{(l)} + \bar{\sigma}_i^{(l)} \normsmall{R_g^{(l)}}{\infty} < 1,
	\end{equation}
	where, for $j \in \{f, i, o\}$, we have:
	\begin{subequations}
		\begin{align}
			\label{eq:sigma_bars}
			&\bar{\sigma}_j^{(l)} = \sigma\left(\norm{\begin{bmatrix} W_j^{(l)} \tilde{\boldsymbol{u}}_{\max}^{(l)} & R_{j}^{(l)} & \boldsymbol{b}_j^{(l)} \end{bmatrix}}{\infty}\right), \\
			\label{eq:bounds_on_the_inputs_of_each_layer}
			&\tilde{\boldsymbol{u}}_{\max}^{(l)} = \begin{cases}
				\boldsymbol{u}_{\max{}} & \mathrm{if} \ l = 1, \\
				\boldsymbol{1}_{n_{\mathrm{hu}}^{(l-1)}} & \mathrm{if} \ l \in \{2, \ldots, L\}.
			\end{cases}
		\end{align}
	\end{subequations}
\end{theorem}
\begin{proof}
	See \proofreference{}.
\end{proof}
\begin{remark}
	\label{rem:diff_from_original_work}
	Compared to the condition for \ISSp{2} of \LSTM{} layers derived in \cite[Proposition 1]{terzi_learning_2021}, which amounts to:
	\begin{equation}
		\label{eq:lstm_stability_condition_ISS2}
		\bar{\sigma}_f^{(l)} + \bar{\sigma}_o^{(l)} \bar{\sigma}_i^{(l)} \normsmall{R_g^{(l)}}{2} < 1,
	\end{equation}
	where $\bar{\sigma}_j^{(l)}, j \in \{f, i, o\},$ are defined as in \eqref{eq:sigma_bars}, the condition for \strictISS{} obtained in this work, i.e. \eqref{eq:lstm_stability_condition}, depends on fewer network parameters. Specifically, it is not related to the parameters $W_o^{(l)}, R_o^{(l)}$, and $\boldsymbol{b}_o^{(l)}$ that are needed for the computation of $\bar{\sigma}_o^{(l)}$ in \eqref{eq:lstm_stability_condition_ISS2}. \proofreference{} reports, in detail, the main differences in the proof for \ISSp{\infty} of an \LSTM{} layer compared to \cite[Proof of Theorem 1]{terzi_learning_2021} (i.e. \ISSp{2}), which lead to \eqref{eq:lstm_stability_condition} instead of \eqref{eq:lstm_stability_condition_ISS2}. 
	In any case, since $\frac{1}{\sqrt{2 n_{\mathrm{hu}}^{(l)}}} \normsmall{\boldsymbol{x}_k^{(l)}}{2} \leq \normsmall{\boldsymbol{x}_k^{(l)}}{\infty} \leq \normsmall{\boldsymbol{x}_k^{(l)}}{2}$ 
	for any $\boldsymbol{x}_k^{(l)} \in \mathbb{R}^{2 n_{\mathrm{hu}}}$, if the condition in \eqref{eq:lstm_stability_condition} is satisfied, then the $l$-th \LSTM{} layer is also \ISSp{2}.
\end{remark}
Additionally, it is possible to demonstrate that an \LSTM{} network with $L$ \strictISS{} \LSTM{} layers 
is \strictISS{}, as claimed in the following Theorem.
\begin{theorem}[\strictISS{} for an \LSTM{} network]
	\label{th:cascaded_lstm_iss_fc}
	An \LSTM{} network as in \eqref{eq:LSTM_model} is \strictISS{} if each \LSTM{} layer that composes it satisfies the condition in \eqref{eq:lstm_stability_condition}.
\end{theorem}
\begin{proof}
	See \proofreference{}.
\end{proof}

\subsection{\strictISS{}-promoted \LSTM{} training}
\label{ss:iss_promoted_lstm_training}
This Section tackles the estimation of the parameters $\theta$ in \eqref{eq:LSTM_parameters}. 
For the model identification purpose, we assume to have at our disposal a set of $N_{\mathrm{e}} \in \mathbb{N}$ sequences of data $\mathcal{D} = \{ \mathcal{D}^{(1)}, \dots, \mathcal{D}^{(N_{\mathrm{e}})}\}$, where each $\mathcal{D}^{(e)} = \{(\boldsymbol{u}_{k}^{(e)}, \boldsymbol{y}_{k}^{(e)}): k \in \{0, \dots, N^{(e)} - 1 \} \}, e \in \{1, \ldots, N_{\mathrm{e}}\}$, consists of input-output data obtained by applying the input sequence $\mathbf{u}^{(e)} = \{\boldsymbol{u}_h^{(e)} \in \mathcal{U}: h \in \{0, \ldots,  N^{(e)} - 1\}\}$ to the system under study. We remark that each $\mathcal{D}^{(e)}$ is composed of $N^{(e)} \in \mathbb{N}$ data in total. We split $\mathcal{D}$ into training, validation, and test datasets, namely $\mathcal{D}_{\mathrm{tr}}, \mathcal{D}_{\mathrm{val}}$, and $\mathcal{D}_{\mathrm{tst}}$ respectively, as follows:
\begin{equation}
	\label{eq:dataset_definition}
	\mathcal{D}_{j} = \{\mathcal{D}^{(e)}: e \in \mathcal{I}_{j}\},
\end{equation}
where $j \in \{\mathrm{tr}, \mathrm{val}, \mathrm{tst}\}$ and $\mathcal{I}_{j} \subseteq \{1, \ldots, N_{\mathrm{e}}\}$ are sets of indexes such that $\mathcal{D}_{\mathrm{tr}} \cup \mathcal{D}_{\mathrm{val}} \cup \mathcal{D}_{\mathrm{tst}} = \mathcal{D}$ and $\mathcal{D}_{\mathrm{tr}} \cap \mathcal{D}_{\mathrm{val}} = \mathcal{D}_{\mathrm{val}} \cap \mathcal{D}_{\mathrm{tst}} = \mathcal{D}_{\mathrm{tr}} \cap \mathcal{D}_{\mathrm{tst}} = \emptyset$. In practice, $\mathcal{D}_{\mathrm{tr}}$ and $\mathcal{D}_{\mathrm{val}}$ are used for model estimation while $\mathcal{D}_{\mathrm{tst}}$ for performance assessment. 
Now, let $\mathcal{I} \subseteq \{1, \ldots, N_{\mathrm{e}}\}$ be a set of indexes such as those in \eqref{eq:dataset_definition}. We define the Mean Squared Error (\MSE{}) over the datasets encompassed by $\mathcal{I}$ as:
\begin{equation}
	\label{eq:MSE}
	\mathrm{\MSE{}}\left(\theta; \mathcal{I}\right)\!=\!\frac{1}{\left|\mathcal{I}\right|} \! \sum_{e \in \mathcal{I}}\left[\frac{1}{N^{(e)}} \!\!\!\! \sum_{k = 0}^{N^{(e)}\!-\!1} \norm{\boldsymbol{y}_{k}^{(e)}\!-\!\hat{\boldsymbol{y}}_{k}^{(e)}\!\left(\theta\right)}{2}^2\right],
\end{equation}
where $\hat{\boldsymbol{y}}_{k}^{(e)}\left(\theta\right)$ is the prediction of \eqref{eq:LSTM_model} at time step $k \in \{0, \dots, N^{(e)} - 1 \}$ simulated using the input sequence $\mathbf{u}^{(e)}, e \in \mathcal{I},$ and parameters $\theta$. Traditionally, $\theta$ in \eqref{eq:LSTM_parameters} is found via the minimization of the training set error $\mathrm{\MSE{}}\left(\theta; \mathcal{I_{\mathrm{tr}}}\right)$. 
However, to ensure that the resulting \LSTM{} network is \strictISS{} as in Theorem \ref{th:cascaded_lstm_iss_fc}, we would need to carry out a constrained minimization of $\mathrm{\MSE{}}\left(\theta; \mathcal{I_{\mathrm{tr}}}\right)$ with $\theta$ satisfying \eqref{eq:lstm_stability_condition} for all $l \in \{1, \ldots, L\}$. Yet, most optimization algorithms for neural networks (such as AdaGrad, RMSProp, and Adam \cite{goodfellow_deep_2016}) are iterative unconstrained gradient-based optimization procedures. Consequently, to promote the \strictISS{} property, we define the stability term
\begin{equation*}
	\label{eq:stability_term}
	\mathrm{\ISS{}}_{\infty}^{(l)}(\theta) = \bar{\sigma}_f^{(l)} + \bar{\sigma}_i^{(l)} \left\| R_g^{(l)} \right\|_\infty - 1
\end{equation*}
according to the condition in \eqref{eq:lstm_stability_condition}, the stability margin $\margin \in [0, 1)$, and re-write the loss function in \eqref{eq:MSE} as:
\begin{equation}
	\label{eq:modified_loss_penalty}
	\ell(\theta) = \mathrm{\MSE{}}\left(\theta; \mathcal{I_{\mathrm{tr}}}\right) + \ISSpenalty \sum_{l=1}^{L} \max\{\mathrm{\ISS{}}_{\infty}^{(l)}(\theta) + \margin, 0\},
\end{equation}
where $\ISSpenalty \in \mathbb{R}_{\geq 0}$ is the penalty weight. The second term of \eqref{eq:modified_loss_penalty} discourages sets of parameters that violate the condition in \eqref{eq:lstm_stability_condition} with some margin. The loss in \eqref{eq:modified_loss_penalty} is inspired by \cite{terzi_learning_2021,bonassi_stability_2021}, where the authors promote the \ISSp{2} and \strictISS{} of \LSTM{} and \GRU{} networks respectively in a similar fashion. 
\begin{remark}
	Intuitively, the derived condition in \eqref{eq:lstm_stability_condition} rather than \eqref{eq:lstm_stability_condition_ISS2} makes network training more stable as less parameters are involved in the penalty term in \eqref{eq:modified_loss_penalty}. 
\end{remark}
In this work, we also modify the traditional early stopping strategy \cite{goodfellow_deep_2016} to ensure that $\theta$ returned by the minimization of \eqref{eq:modified_loss_penalty} is such that \eqref{eq:lstm_stability_condition} holds $\forall l \in \{1, \ldots, L\}$, i.e. the \LSTM{} in \eqref{eq:LSTM_model} is \strictISS{}. Essentially, \LSTM{} training is an iterative procedure that works as follows: starting from an initial set of parameters $\theta_0$, at each iteration a gradient-based optimizer updates $\theta$ based on the gradient of \eqref{eq:modified_loss_penalty}, namely $\nabla \ell(\theta)$, and a learning rate $\learnrate \in \mathbb{R}_{\geq 0}$. Every $\kappa_{\mathrm{val}} \in \mathbb{N}$ iterations, we compute the \MSE{} in \eqref{eq:MSE} on the validation set $\mathcal{D}_{\mathrm{val}}$ 
and assess if the \LSTM{} at that iteration is \strictISS{}. If the validation error improves while \eqref{eq:lstm_stability_condition} holds $\forall l \in \{1, \ldots, L\}$, we store that set of parameters. The training procedure is stopped once the $\mathrm{\MSE{}}\left(\theta; \mathcal{I_{\mathrm{val}}}\right)$ fails to improve for more than $p_{\mathrm{val}} \in \mathbb{N}$ validation checks or once a maximum number of iterations $\kappa_{\max} \in \mathbb{N}$ is reached, returning the set of parameters with the best performance on $\mathcal{D}_{\mathrm{val}}$, which we denote as $\theta^*$. 
\begin{remark}
	The proposed training strategy can be applied to any \RNN{} model for which a condition for \ISS{}, such as \eqref{eq:lstm_stability_condition}, is available.
\end{remark}
\section{Thermal system case study}
\label{s:thermal_system_under_study}
To assess the effectiveness of the methodology proposed in Section \ref{s:input_to_state_stable_lstm_networks}, we consider a thermal system case study. Specifically, we model the shrink tunnel treated in \cite{pitturelli_towards_2024}, a machine that is commonly used for bottle packaging purposes. 
The shrink tunnel, depicted in 
Figure \ref{fig:shrink_tunnel_scheme},
features an industrial oven and a conveyor belt. Bottles are mechanically grouped together and wrapped in a thin plastic film before entering the oven. As bottles move through the heated environment, the plastic film shrinks, tightly enclosing the bottles to form the packs. An infrared sensor detects the presence or absence of bottles at the tunnel entrance.
The oven cavity is divided into two interconnected heating zones. Twelve thermocouples in total, installed as reported in 
Figure \ref{fig:shrink_tunnel_scheme},
monitor the temperatures inside the cavity. Twelve Heat Resistors (\HR{}s), connected to the electrical grid via two Solid-State Relays (\SSR{}s) and two ElectroMechanical Relays (\EMR{}s) as in 
Figure \ref{fig:shrink_tunnel_scheme}, produce the heat needed to raise the oven temperature. In detail, the \SSR{}s and \EMR{}s modulate the voltage across the \HR{}s via Pulse Width Modulation (\PWM{}) according to the duty cycles supplied by a temperature controller and with a \PWM{} period $T_{\mathrm{P}} = 30\,\mathrm{s}$. Lastly, four convection fans with shared operating frequency promote the circulation of hot air inside the cavity. 

\begin{figure}[!htb]
	\centering
	\subfloat[Front view.]{
		\centering
		\includegraphics[width=\columnwidth]{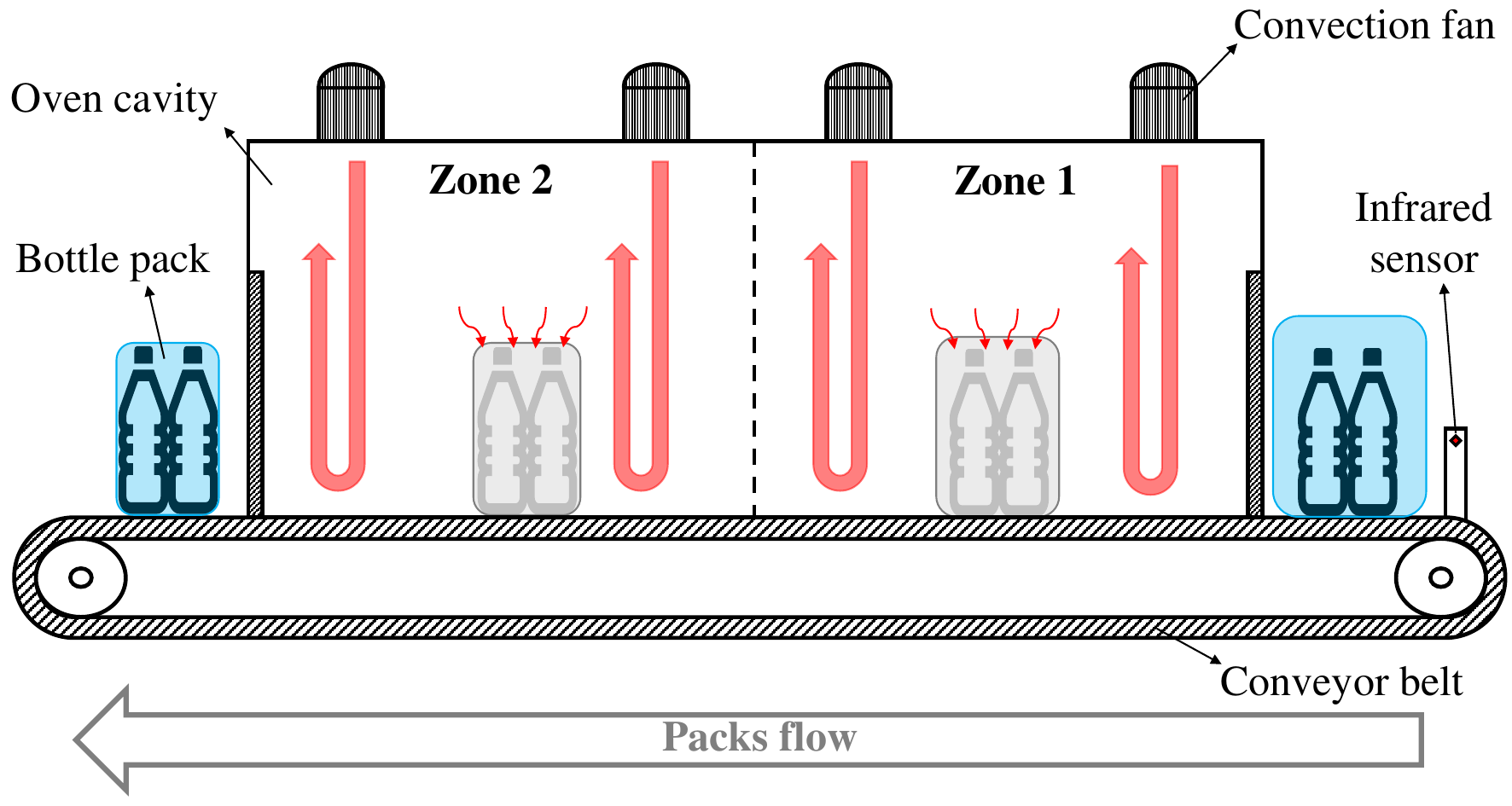}
	}
	
	\subfloat[Top view.]{
		\centering
		\includegraphics[width=\columnwidth]{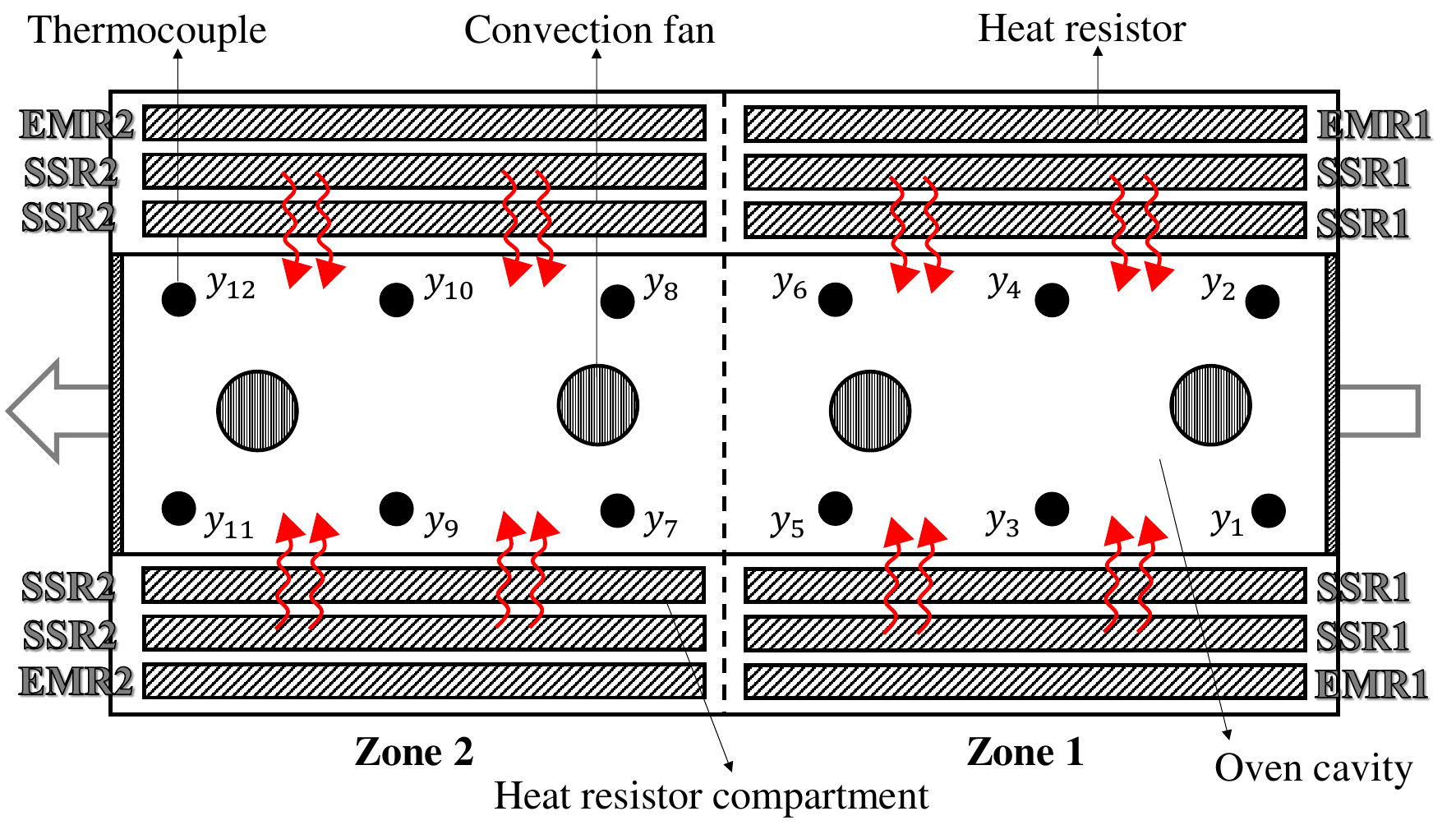}
	}
	\caption{
		\label{fig:shrink_tunnel_scheme}
		Schematic of the considered shrink tunnel. In zone $z, z \in \{1, 2\}$, the heat resistors marked by \SSR{}$z$ are driven by the same solid-state relay while those denoted by \EMR{}$z$ are managed by the electromechanical relay.}
	\vspace{-1em}
\end{figure}
Let $z \in \{1, 2\}$ be the zone of belonging and $r \in \{1, 2\}$ denote the relay type ($r = 1$ for the \SSR{}s and $r = 2$ for the \EMR{}s). The measurable continuous-time signals for the system under study are:
\begin{itemize}
	\item The temperatures measured by the thermocouples $y_{\iota}(t) \in \mathbb{R}$ (in $^\circ \mathrm{C}$), $\iota \in \{1, \ldots, 12\}$;
	\item The duty cycles for the heat resistors in zone $z$ driven by relay $r$, i.e. $w_{z}^{(r)}(t) \in [0, 1]$;
	\item The grid voltage $V_{\mathrm{g}}(t) \in \mathbb{R}_{>0}$ (in $\mathrm{V}$);
	\item The reading of the infrared sensor at the entrance of the oven $d_{\mathrm{p}}(t) \in \{0, 1\}$ ($d_{\mathrm{p}}(t) = 0$ denotes the absence of bottle packs and viceversa for $d_{\mathrm{p}}(t) = 0$);
	\item The convection fans operating frequency $d_{\mathrm{f}}(t) \in [40, 60]$ (in $\mathrm{Hz}$).
\end{itemize}
Each of these signals are sampled at a sampling time $T_{\mathrm{s}} = T_{\mathrm{P}} = 30\,\mathrm{s}$. Overall, the $n_u = 7$ inputs and $n_y = 12$ outputs for the system under study are:
\begin{subequations}
	\label{eq:inputs_outputs_shrink_tunnel}
	\begin{align}
		\boldsymbol{u} &= \begin{bmatrix}
			w_1^{(1)} & w_1^{(2)} & w_2^{(1)} & w_2^{(2)} & V_{\mathrm{g}} & d_{\mathrm{p}} & d_{\mathrm{f}}
		\end{bmatrix}^\top, \\
		\boldsymbol{y} &= \begin{bmatrix}
			y_1 & \ldots & y_{12}
		\end{bmatrix}^\top.
	\end{align}
\end{subequations}

\subsection{Physics-based model}
\label{ss:grey_box_thermal_model}
As a baseline for comparison, we consider a grey-box physics-based model for the shrink tunnel under study. We assume that the temperatures $y_\iota(t), \iota \in \{1, \ldots, n_y\},$ result from the sum of three contributions:
\begin{equation}
	\label{eq:physics_temperature_contributions}
	\boldsymbol{y}(t) = \boldsymbol{T}(t) + \delta\boldsymbol{T}_{\mathrm{p}}(t) + \delta\boldsymbol{T}_{\mathrm{f}}(t),
\end{equation}
where $\boldsymbol{T}(t), \delta\boldsymbol{T}_{\mathrm{p}}(t)$, and $\delta\boldsymbol{T}_{\mathrm{f}}(t)$ are due to the heat produced by the heat resistors (and ambient temperature), the flow of bottle packs inside the oven cavity, and the fan frequency respectively. $\boldsymbol{T}(t)$ is modeled via an electro-equivalent thermal circuit as in \cite[Fig. 1b]{previtali_continuous_2024}, leading to:
\begin{equation}
	\label{eq:T_mu}
	\dot{\boldsymbol{T}}(t) = A_{TT} \boldsymbol{T}(t) + B_q \boldsymbol{q}(t) + \boldsymbol{b}_{T} T_{\mathrm{a}}(t),
\end{equation}
where $\boldsymbol{q}(t) = \begin{bmatrix} q_1(t) & q_2(t)\end{bmatrix}^\top \in \mathbb{R}_{\geq 0}^2$ (in $\frac{\mathrm{J}}{\mathrm{s}}$) are the heat flow rates produced by the heat resistors in each zone and $T_{\mathrm{a}}(t) \in \mathbb{R}$ (in $^\circ \mathrm{C}$) is the ambient temperature. The derivation of the matrices and vectors $A_{TT} \in \mathbb{R}^{n_y \times n_y}, B_q \in \mathbb{R}^{n_y \times 2}$, and $\boldsymbol{b}_{T} \in \mathbb{R}^{n_y}$ is out of scope of this paper, the reader is referred to \cite[Section 3.1]{previtali_continuous_2024} for further details. In any case, $A_{TT}, B_q,$ and $\boldsymbol{b}_{T}$ depend on several thermal resistances and the thermal capacitance of the oven cavity, which are model parameters. Due to Joule heating, the heat flow rates $q_z(t), z \in \{1, 2\},$ are given by:
\begin{equation}
	\label{eq:heat_flow_rates}
	q_z(t)\!=\!\frac{1}{R_{\mathrm{heat}}}\!\left[2 V_z^{(1)}(t; w_{z}^{(1)}\!\!, V_{\mathrm{g}})^2\!+\!V_z^{(2)}(t; w_{z}^{(2)}\!\!, V_{\mathrm{g}})^2\right],
\end{equation}
where $R_{\mathrm{heat}} \in \mathbb{R}_{>0}$ (in $\Omega$) is the ohmic resistance of the heat resistors and $V_z^{(r)}(t; \cdot) \in \mathbb{R}_{\geq 0}$ (in $\mathrm{V}$) are the voltages across the \HR{}s in zone $z$ driven by relay $r \in \{1, 2\}$. The $V_z^{(r)}(t; \cdot)$'s depend on the duty cycles $w_{z}^{(r)}(t)$, due to \PWM{}, and the grid voltage $V_{\mathrm{g}}(t)$. Let $\boldsymbol{V}_{\mathrm{sq}}(t; \cdot) \in \mathbb{R}_{\geq 0}^{4}$ (in $\mathrm{V}^2$) and $\boldsymbol{w}(t) \in [0, 1]^4$ be the vectors of squared grid voltages and duty cycles respectively. Then, we model the propagation of heat from the \HR{}s to the oven cavity via a first-order low-pass filter with unitary gain \cite{previtali_continuous_2024}:
\begin{equation}
	\label{eq:filtered_heat_flow_rates}
	\dot{\boldsymbol{q}}^{(\mathrm{f})}(t) = A_{qq} \boldsymbol{q}^{(\mathrm{f})}(t) + B_{q} \boldsymbol{\boldsymbol{V}}_{\mathrm{sq}}(t; \boldsymbol{w}, V_{\mathrm{g}}),
\end{equation}
where $\boldsymbol{q}^{(\mathrm{f})}(t) \in \mathbb{R}_{\geq 0}$ are the filtered signals, 
$A_{qq} = -\diag{\frac{1}{\tau_{q,1}}, \frac{1}{\tau_{q,2}}}$, 
$\tau_{q,z} \in \mathbb{R}_{\geq 0}$ (in $\mathrm{s}$), $z \in \{1, 2\}$, being the time constants of the filters, and 
$B_{q} = - \frac{1}{R_{\mathrm{heat}}} A_{qq} \begin{bmatrix}
	2 & 1 & 0 & 0 \\
	0 & 0 & 2 & 1
\end{bmatrix}$ (see \eqref{eq:heat_flow_rates}). Then, in \eqref{eq:T_mu}, we replace $\boldsymbol{q}(t)$ with $\boldsymbol{q}^{(\mathrm{f})}(t)$.

Finally, for what concerns $\delta\boldsymbol{T}_{\mathrm{p}}(t)$ and $\delta\boldsymbol{T}_{\mathrm{f}}(t)$ in \eqref{eq:physics_temperature_contributions}, these are modeled via first-order low-pass filters as follow:
\begin{subequations}
	\label{eq:disturbance_eqs}
	\begin{align}
		\delta \dot{\boldsymbol{T}}_{\mathrm{p}}(t) &= A_{\mathrm{pp}} \delta \boldsymbol{T}_{\mathrm{p}}(t) + \boldsymbol{b}_{\mathrm{p}} d_{\mathrm{p}}(t), \\
		\delta \dot{\boldsymbol{T}}_{\mathrm{f}}(t) &= A_{\mathrm{ff}} \delta \boldsymbol{T}_{\mathrm{f}}(t) + \boldsymbol{b}_{\mathrm{f}} d_{\mathrm{f}}(t), 
	\end{align}
\end{subequations}
where $A_{\mathrm{pp}}, A_{\mathrm{ff}} \in \mathbb{R}^{n_y \times n_y}$ and $\boldsymbol{b}_{\mathrm{p}}, \boldsymbol{b}_{\mathrm{f}} \in \mathbb{R}^{n_y}$. In particular, 
$A_{\mathrm{pp}} = -\diag{\frac{1}{\tau_{\mathrm{p}, 1}}^{-1}, \ldots, \frac{1}{\tau_{\mathrm{p},n_y}}}$, $\tau_{\mathrm{p},\iota} \in \mathbb{R}_{\geq 0}$ (in $\mathrm{s}$), $\iota \in \{1, \ldots, n_y\}$, being the time constants of the filters, and $\boldsymbol{b}_{\mathrm{p}} = - A_{\mathrm{pp}} \begin{bmatrix}
	\mu_{\mathrm{p},1} & \cdots & \mu_{\mathrm{p},n_y}
\end{bmatrix}^\top$, where $\mu_{\mathrm{p},\iota} \in \mathbb{R}$ (in $^\circ \mathrm{C}$) are the gains. $A_{\mathrm{ff}}$ and $\boldsymbol{b}_{\mathrm{f}}$ are defined in a similar fashion. The relationships in \eqref{eq:disturbance_eqs} are derived from experimental insights. In practice, the insertion of bottle packs causes the oven temperatures to lower over time until reaching an equilibria. Similarly, a change in the operating frequency of the fans shifts the temperature equilibria with certain dynamics.

\paragraphfont{Discretization}
The just-derived continuous-time shrink tunnel model can be discretized following the approach proposed in \cite[Section 3.2]{previtali_grey-box_2024}, leading to the following nonlinear discrete-time state-space model:

\begin{equation}
	\label{eq:discrete-time_grey-box_model}
	\begin{cases}
		\boldsymbol{q}_{k+1}^{(\mathrm{f})} = \tilde{A}_{qq} \boldsymbol{q}_{k}^{(\mathrm{f})} + \tilde{B}_{q} V_{\mathrm{g}, k}^2 \boldsymbol{w}_k, \\
		\boldsymbol{T}_{k+1} = \tilde{A}_{TT} \boldsymbol{T}_{k} + \tilde{B}_{q} \boldsymbol{q}_{k}^{(\mathrm{f})} + \tilde{\boldsymbol{b}}_T T_{\mathrm{a}, k}, \\
		\delta \boldsymbol{T}_{\mathrm{p}, k+1} = \tilde{A}_{\mathrm{pp}} \delta \boldsymbol{T}_{\mathrm{p}, k} + \tilde{\boldsymbol{b}}_{\mathrm{p}} d_{\mathrm{p}, k}, \\
		\delta \boldsymbol{T}_{\mathrm{f}, k+1} = \tilde{A}_{\mathrm{ff}} \delta \boldsymbol{T}_{\mathrm{f}, k} + \tilde{\boldsymbol{b}}_{\mathrm{f}} d_{\mathrm{f}, k}, \\
		\boldsymbol{y}_k = \boldsymbol{T}_{k} + \delta \boldsymbol{T}_{\mathrm{p}, k} + \delta \boldsymbol{T}_{\mathrm{f}, k},
	\end{cases}
\end{equation}
where the matrices and vectors highlighted with a $\tilde{\cdot}$ result from the discretization of their continuous-time counterparts. Overall, the model in \eqref{eq:discrete-time_grey-box_model} has $n_x = 38$ states and $74$ parameters, i.e. the thermal resistances for \eqref{eq:T_mu} (see \cite{previtali_continuous_2024}), the ohmic resistance $R_{\mathrm{heat}}$ in \eqref{eq:heat_flow_rates}, and the time constants and gains for the filters in \eqref{eq:filtered_heat_flow_rates} and \eqref{eq:disturbance_eqs}.
\begin{remark}
	\label{rem:ISS_thermal_system}
	The continuous-time model composed of \eqref{eq:T_mu}, \eqref{eq:filtered_heat_flow_rates}, and \eqref{eq:disturbance_eqs} is input-to-state stable as in \cite{sontag_input--state_1995} since it is the composition of \ISS{} models (i.e., the electro-equivalent thermal circuit model in \cite[Fig. 1b]{previtali_continuous_2024} and three stable low-pass filters). However, due to the discretization, \eqref{eq:discrete-time_grey-box_model} may not be \strictISS{} as in Definition \ref{def:strict_iss}. In any case, by introducing a new input $\boldsymbol{w}_{\mathrm{V}, k} = V_{\mathrm{g}, k}^2 \boldsymbol{w}_k$, the model in \eqref{eq:discrete-time_grey-box_model} becomes linear in $\boldsymbol{u}_{\mathrm{V}, k} = \begin{bmatrix}
		\boldsymbol{w}_{\mathrm{V}, k}^\top & T_{\mathrm{a}, k} & d_{\mathrm{p}, k} & d_{\mathrm{f}, k}
	\end{bmatrix}^\top$ and its stability can be analyzed by checking the eigenvalues of the state matrix.
\end{remark}
\section{Experimental results}
\label{s:experimental_results}
This Section analyzes the accuracy of the \strictISS{}-promoted \LSTM{} in \eqref{eq:LSTM_model}, the \strictISS{}-promoted \GRU{} in \cite{bonassi_stability_2021}, and the grey-box physics-based model in \eqref{eq:discrete-time_grey-box_model} on experimental data coming from the shrink tunnel described in Section \ref{s:thermal_system_under_study}. For completeness, the \ISS{}-promoted-\RNN{}s are also compared to \RNN{}s identified via traditional training strategies.

\paragraphfont{Experimental setup}
A total of $N_{\text{e}} = 12$ experiments were carried out on the shrink tunnel under study, each trial lasting between $2$ and $7.5$ hours.
The experiments encompass a variety of operating conditions for the thermal system, including the temperature responses due to the application of constant duty cycles and pseudo-random binary sequences, 
the insertion of bottle packs inside the oven cavity at different conveyor belt speeds, and data related to the closed-loop operation of the system under study. We split the data into training, validation, and test datasets (Section \ref{s:problem_statement}) in such a way that $\left|\mathcal{D}_{\mathrm{tr}}\right| = 9, \left|\mathcal{D}_{\mathrm{val}}\right| = 2$, and $\left|\mathcal{D}_{\mathrm{tst}}\right| = 1$, making sure that each dataset contains at least a sequence where all the inputs in \eqref{eq:inputs_outputs_shrink_tunnel} vary\footnote{We point out that the ambient temperature stays constant throughout each experiment and, consequently, we can remove it from the temperature measures used for identification purposes and add it back a-posteriori. This rationale is also motivated by the fact that $T_{\mathrm{a}}(t)$ is not measured directly.}.

\paragraphfont{\LSTM{} and \GRU{} training}
As customary, each signal in \eqref{eq:inputs_outputs_shrink_tunnel} is normalized before network training \cite{goodfellow_deep_2016}, making $\boldsymbol{u}_{\mathrm{max}} = \boldsymbol{1}_{n_u}$ in \eqref{eq:input_bound}.
Afterwards, we estimate $\theta$ in \eqref{eq:LSTM_parameters} according to Section \ref{ss:iss_promoted_lstm_training} with penalty coefficient $\ISSpenalty = 0.05$, stability margin $\margin = 0.05$, maximum number of iterations $\kappa_{\max} = 2500$, $\kappa_{\mathrm{val}} = 25$, and validation patience $p_{\mathrm{val}} = 20$. The number of layers for the \LSTM{} in \eqref{eq:LSTM_model} is set to $L = 3$. Instead, $n_{\mathrm{hu}}^{(l)}, l \in \{1, \ldots, L\},$ and the learning rate $\learnrate$ are tuned via Bayesian optimization \cite{frazier_tutorial_2018}, by finding the set of hyper-parameters $\{n_{\mathrm{hu}}^{(1)}, \ldots, n_{\mathrm{hu}}^{(L)}, \learnrate\}$ that minimize the final validation \MSE{} attained during training as in Section \ref{ss:iss_promoted_lstm_training}. The Adam algorithm is used for network training \cite{goodfellow_deep_2016}. The same approach is used for a \GRU{} model with $L = 3$ layers and using the \strictISS{} condition in \cite[Equation (8)]{bonassi_stability_2021} rather than \eqref{eq:lstm_stability_condition}. Table \ref{tab:optimized_hyperparameters} reports the optimized hyper-parameters along with the total number of parameters $\theta$ of each network. We point out that the resulting \LSTM{} and \GRU{} are both \strictISS{} according to Definition \ref{def:strict_iss}. To evaluate the impact of \ISS{} promotion in \eqref{eq:modified_loss_penalty} on performance, we also train \LSTM{} and \GRU{} networks following traditional training strategies, i.e. by minimizing the $\mathrm{\MSE{}}\left(\theta; \mathcal{I_{\mathrm{tr}}}\right)$ in \eqref{eq:MSE}, using the hyper-parameters in Table \ref{tab:optimized_hyperparameters}. In this case, the resulting \RNN{}s are not \strictISS{} since \eqref{eq:lstm_stability_condition} and \cite[Equation (8)]{bonassi_stability_2021} do not hold.
\begin{table}[!htb]
	\centering
	\settablestretch
	\settablefontsize
	\caption{\label{tab:optimized_hyperparameters} Optimized hyper-parameters for the \LSTM{} and \GRU{} networks.}
	\begin{tabular}{cccccc}
		Model                   & $n_{\mathrm{hu}}^{(1)}$ & $n_{\mathrm{hu}}^{(2)}$ & $n_{\mathrm{hu}}^{(3)}$ & $\learnrate$       & \#parameters\tabularnewline
		\hline
		\colorrowoftable LSTM{} & $88$                    & $33$                    & $68$                    & $5.66\cdot10^{-3}$ & $7.84\cdot10^{4}$\tabularnewline
		GRU{}                   & $497$                   & $37$                    & $142$                   & $1.78\cdot10^{-3}$ & $8.91\cdot10^{5}$\tabularnewline
	\end{tabular}
\end{table}

\paragraphfont{Identification of the physics-based model}
Similarly to the \RNN{}s, the parameters of the physics-based model in \eqref{eq:discrete-time_grey-box_model} are estimated via the minimization of $\mathrm{\MSE{}}\left(\theta; \mathcal{I_{\mathrm{tr}}} \cup \mathcal{I_{\mathrm{val}}}\right)$ in \eqref{eq:MSE}, i.e. using both the training and validation datasets (no need for early stopping as in Section \ref{ss:iss_promoted_lstm_training}). \strictISS{} is checked a-posteriori as mentioned in Remark \ref{rem:ISS_thermal_system}, assessing that the resulting model satisfies that property.

\paragraphfont{Results}
Now, we evaluate the performances of the models on test data. 
We analyze the fits for each temperature $y^{(e)}_\iota, \iota \in \{1, \ldots, n_y\}$, $e \in \mathcal{I}_{\mathrm{tst}}$, which are defined as follows:
\begin{equation}
	\label{eq:fits}
	\mathrm{Fit}_{\iota}^{(e)} = 1 - \frac{\sqrt{\frac{1}{N^{(e)}}\sum_{k = 0}^{N^{(e)}-1} \left|y_{\iota, k}^{(e)} - \hat{y}_{\iota, k}^{(e)}(\theta^*)\right|^2}}{\max_k{y_{\iota, k}^{(e)}} - \min_k{y_{\iota, k}^{(e)}}},
\end{equation}
where $\hat{y}_{\iota, k}^{(e)}(\theta^*)$ are the identified models predictions. \figurenametext{} \ref{fig:fit_box_chart} illustrates the distributions of the fits for the different models. The \strictISS{} \LSTM{} and \strictISS{} \GRU{} achieve the best performance, with the former attaining a slightly higher median fit. Notably, promoting \strictISS{} leads to improved results compared to non-\strictISS{}-promoted training, suggesting that stability promotion contributes to better generalization. In contrast, the physics-based model underperforms relative to all \RNN{}s, exhibiting lower median fit and higher dispersion.
\begin{figure}[!htb]
	\centering
	\includegraphics[width=\columnwidth]{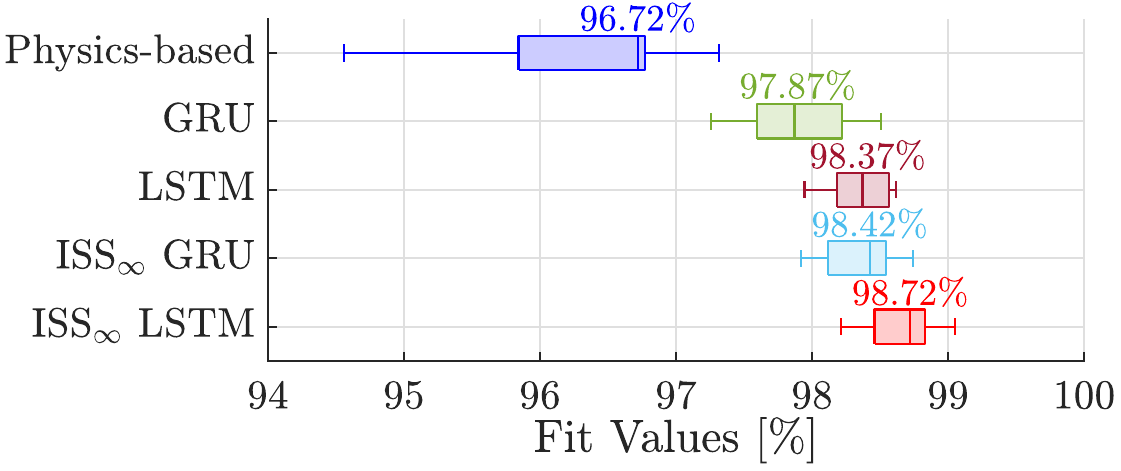}
	
	\caption{
		\label{fig:fit_box_chart}
		Box plot of the performance index in \eqref{eq:fits} for each model on test data. The Figure also reports the median fits.
	}
\end{figure}

As an example, \figurenametext{} \ref{fig:temperature} shows the test trial case for temperature $y_3$ over time, along with the inputs, for the \strictISS{} networks and the physics-based model. The \strictISS{} \LSTM{} and \strictISS{} \GRU{} networks closely track the real temperature, especially during fan frequency ($d_{\mathrm{f}}$) variations and pack disturbances ($d_{\mathrm{p}}$). Instead, the physics-based model only captures accurately the temperature response in the initial part of the experiment, between $0\,\mathrm{min}$ and $120\,\mathrm{min}$, which depends only on the duty cycles ($\boldsymbol{w}$) and grid voltage ($V_{\mathrm{g}}$).
\begin{figure}[!htb]
	\centering
	\includegraphics[width=\columnwidth]{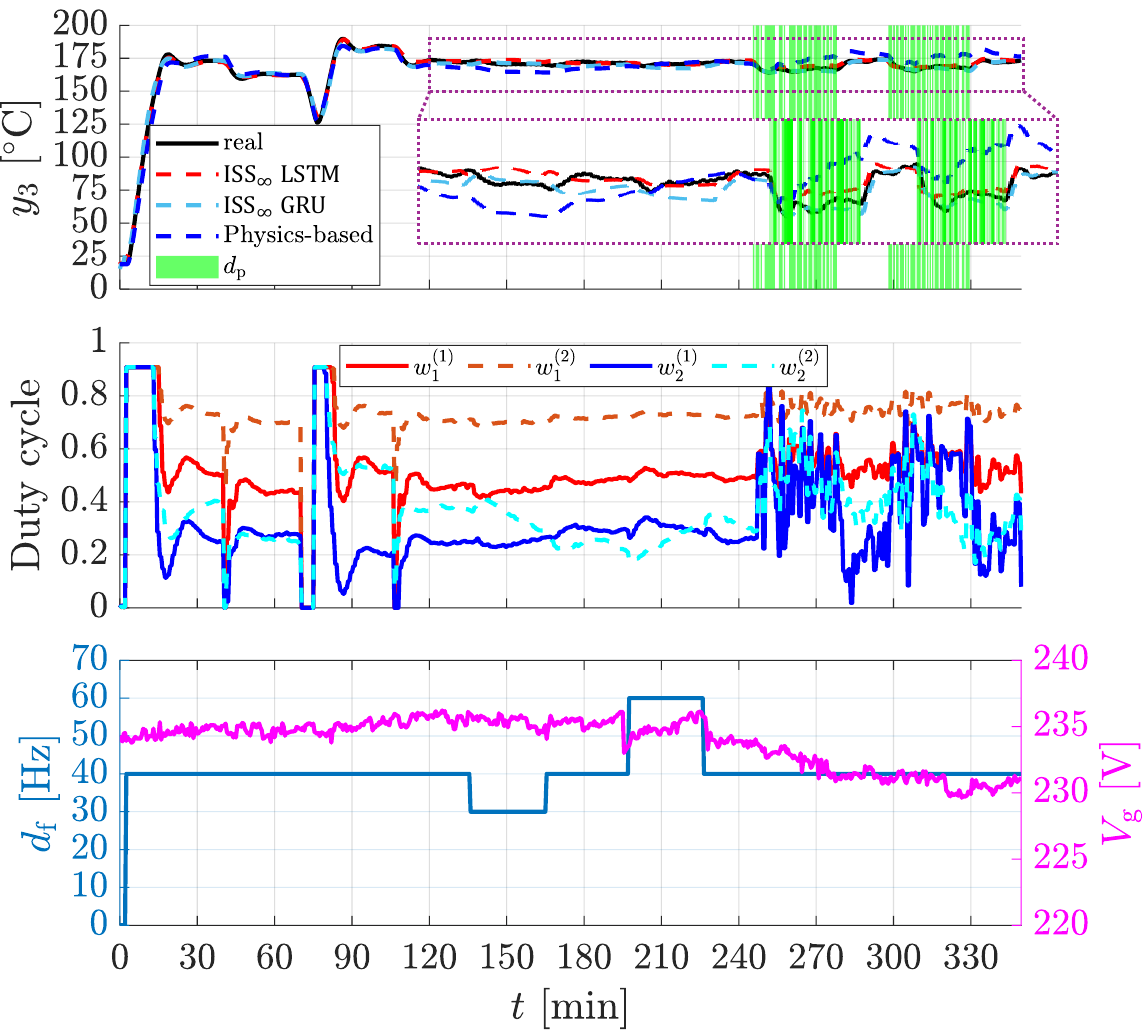}
	\caption{
		\label{fig:temperature}
		Comparison of temperature $y_3$ predictions on the test trial focusing on fan frequency changes and pack disturbances. The green vertical stripes denote when $d_{\mathrm{p}, k} = 1$.
	}
\end{figure}

In terms of model size, the \LSTM{}s require around a tenth of the \GRU{}s parameters (see Table \ref{tab:optimized_hyperparameters}), making them preferable both in terms of performance and complexity. However, both network models have significantly more parameters than the physics-based one, which only uses $74$ (see Section \ref{ss:grey_box_thermal_model}). Regardless, it is worth noticing that the derivation of the model in \eqref{eq:discrete-time_grey-box_model} requires a cumbersome and time-consuming engineering phase involving, e.g., the design of the electro-equivalent thermal circuit in \cite[Fig. 1b]{previtali_continuous_2024}. In contrast, the networks, while offering superior accuracy, rely entirely on data-driven training, making them sensitive to the quality of the data and potentially requiring more experiments to achieve accurate generalization.
\section{Conclusion}
\label{s:conclusion}
In this paper, we derive a sufficient condition to guarantee \strictISS{} of \LSTM{} networks which, compared to \cite[Proposition 1]{terzi_learning_2021} for \ISSp{2}, depends on fewer parameters. Then, we propose a training strategy with early stopping to promote input-to-state stability of recurrent neural networks. We later apply the proposed \strictISS{}-promoted \LSTM{} network on a thermal system case study, i.e. the identification of a shrink tunnel model from real data. The performances of the \strictISS{} \LSTM{} are compared against a \strictISS{} \GRU{} \cite{bonassi_stability_2021}, non-\strictISS{}-promoted \RNN{}s, and a physics-based grey-box thermal model. Experimental results show that the \strictISS{} \LSTM{} outperforms all other models in terms of predictive accuracy while requiring minimal prior knowledge of the system. 

\section*{Acknowledgments}
This work was partially funded by the EU European Defence Fund Project KOIOS (EDF-2021-DIGIT-R-FL-KOIOS). KOIOS is a project Funded by the European Union. Views and opinions expressed are however those of the author(s) only and do not necessarily reflect those of the European Union or European Defence Fund. Neither the European Union nor the granting authority can be held responsible for them.

\bibliographystyle{plain}
{
    \bibliography{bibliography.bib}
}

\appendix
\renewcommand{\appendixname}{Appendix}
\renewcommand{\thesection}{\appendixname~\Alph{section}}
\section{Proofs for the Theorems}
\label{s:appendix_proofs}
\paragraphfont{Proof of Theorem \ref{th:lstm_iss}}
Firstly, recall that the sigmoid and hyperbolic tangent functions are strictly monotonic and the following bounds hold for any $x \in \mathbb{R}$:
\begin{equation}
	\label{eq:sigma_tanh_bound}
	\sigma(x) \in (0,1), \quad \tanh(x) \in (-1,1).
\end{equation}
Consequently, due to \eqref{eq:gates_LSTM} and Assumption \ref{ass:initial_hidden_states}, we can bound the hidden state in \eqref{eq:LSTM_layer_hidden_state_update} for the $l$-th layer, $l \in \{1, \ldots, L\}$, of the \LSTM{} network as:
\begin{equation}
	\label{eq:hidden_state_bound}
	\boldsymbol{h}_k^{(l)} \in (-1,1)^{n_{\mathrm{hu}}^{(l)}}, \quad 
	\forall k \in \mathbb{N}.
\end{equation}
Then, due to \eqref{eq:hidden_state_bound} and Assumption \ref{as:boundedness_of_input}, we can bound the inputs for each layer (i.e. $\tilde{\boldsymbol{u}}_k^{(l)}$ in \eqref{eq:input_for_each_layer}) according to
\begin{equation}
	\label{eq:input_bound_for_each_layer}
	\tilde{\boldsymbol{u}}_k^{(l)} \in \tilde{\mathcal{U}}^{(l)} = \{\tilde{\boldsymbol{u}}: - \tilde{\boldsymbol{u}}_{\max}^{(l)} \leq \tilde{\boldsymbol{u}} \leq \tilde{\boldsymbol{u}}_{\max}^{(l)}\}, \quad \forall k \in \mathbb{N},
\end{equation}
where $\tilde{\boldsymbol{u}}_{\max}^{(l)}$ is defined as in \eqref{eq:bounds_on_the_inputs_of_each_layer}. Following \cite{terzi_learning_2021}, due to the properties and range of the sigmoid function in \eqref{eq:sigma_tanh_bound}, as well as the bounds in \eqref{eq:input_bound_for_each_layer}, we can find an upper bound for each component of the forget, input, and output gates in \eqref{eq:gates_LSTM}, i.e. $f_{\iota, k}^{(l)}, i_{\iota, k}^{(l)},$ and $o_{\iota, k}^{(l)}, \iota \in \{1, \ldots, n_{\mathrm{hu}}^{(l)}\},$ respectively, as:
\begin{equation}
	\label{eq:gate_bounds}
	|j_{\iota, k}^{(l)}| \leq \sigma\left(\norm{\begin{bmatrix} W_j^{(l)} \tilde{\boldsymbol{u}}_{\max}^{(l)} & R_{j}^{(l)} & \boldsymbol{b}_j^{(l)} \end{bmatrix}}{\infty}\right) = \bar{\sigma}_j^{(l)},
\end{equation}
where $j \in \{f, i, o\}$. Now, consider the cell state update in \eqref{eq:LSTM_layer_cell_state_update}. Taking the infinity norm and applying the triangle inequality, we get:
\begin{equation}
	\label{eq:cell_state_bound}
	\normsmall{\boldsymbol{c}_{k}^{(l)}}{\infty} \leq \normsmall{\boldsymbol{f}_{k-1}^{(l)}}{\infty} \normsmall{\boldsymbol{c}_{k-1}^{(l)}}{\infty} + \normsmall{\boldsymbol{i}_{k-1}^{(l)}}{\infty} \normsmall{\boldsymbol{g}_{k-1}^{(l)}}{\infty}.
\end{equation}
Using the bounds in \eqref{eq:gate_bounds} and recalling that $\norm{\boldsymbol{\tanh}\left(\boldsymbol{x}\right)}{\infty} \leq \norm{\boldsymbol{x}}{\infty}$ for any $\boldsymbol{x} \in \mathbb{R}^{n_x}$ due to the 1-Lipschitz continuity of the hyperbolic tangent function, it is possible to prove that:
\begin{align}
	\label{eq:cell_state_bound_simplified}
	\normsmall{\boldsymbol{c}_{k}^{(l)}}{\infty} & \leq \bar{\sigma}_f^{(l)} \normsmall{\boldsymbol{c}_{k-1}^{(l)}}{\infty} + \bar{\sigma}_i^{(l)} \normsmall{W_g^{(l)}}{\infty} \normsmall{\tilde{\boldsymbol{u}}_{k-1}^{(l)}}{\infty} + \nonumber \\
	& \  + \bar{\sigma}_i^{(l)} \normsmall{R_g^{(l)}}{\infty} \normsmall{\boldsymbol{h}_{k-1}^{(l)}}{\infty}
	+ \bar{\sigma}_i^{(l)} \normsmall{\boldsymbol{b}_g^{(l)}}{\infty}.
\end{align}
Similarly, for the hidden state in \eqref{eq:LSTM_layer_hidden_state_update}, we have:
\begin{equation}
	\label{eq:hidden_state_bound_simplified}
	\normsmall{\boldsymbol{h}_{k}^{(l)}}{\infty} \leq \bar{\sigma}_o^{(l)} \norm{\boldsymbol{\tanh}\left(\boldsymbol{c}_{k}^{(l)} \right)}{\infty} \leq \bar{\sigma}_o^{(l)} \normsmall{\boldsymbol{c}_{k}^{(l)}}{\infty}.
\end{equation}
By combining \eqref{eq:cell_state_bound_simplified} and \eqref{eq:hidden_state_bound_simplified}, we get:
\begin{equation}
	\begin{bmatrix}
		\normsmall{\boldsymbol{c}_{k}^{(l)}}{\infty}\! \\
		\normsmall{\boldsymbol{h}_{k}^{(l)}}{\infty}\!
	\end{bmatrix}
	\!\leq\!
	\mathbf{A}^{\!(l)}\!
	\begin{bmatrix}
		\normsmall{\boldsymbol{c}_{k-1}^{(l)}}{\infty}\! \\
		\normsmall{\boldsymbol{h}_{k-1}^{(l)}}{\infty}\!
	\end{bmatrix}
	+
	\mathbf{B}_u^{(l)} \normsmall{\tilde{\boldsymbol{u}}_{k-1}^{(l)}}{\infty}
	+
	\mathbf{B}_b^{(l)} \normsmall{\boldsymbol{b}_g^{(l)}}{\infty},
	\label{eq:state_update_matrix}
\end{equation}
where:
\begin{equation*}
	\label{eq:AB_definitions}
	\begin{aligned}
		& \mathbf{A}^{\!(l)} =
		\begin{bmatrix}
			\bar{\sigma}_f^{(l)}                      & \bar{\sigma}_i^{(l)} \normsmall{R_g^{(l)}}{\infty}                      \\
			\bar{\sigma}_o^{(l)} \bar{\sigma}_f^{(l)} & \bar{\sigma}_o^{(l)} \bar{\sigma}_i^{(l)} \normsmall{R_g^{(l)}}{\infty}
		\end{bmatrix}, \\
		\mathbf{B}_u^{(l)} =
		& \begin{bmatrix}
			\bar{\sigma}_i^{(l)} \normsmall{W_g^{(l)}}{\infty} \\
			\bar{\sigma}_o^{(l)} \bar{\sigma}_i^{(l)} \normsmall{W_g^{(l)}}{\infty}
		\end{bmatrix},
		\
		\mathbf{B}_b^{(l)} =
		\begin{bmatrix}
			\bar{\sigma}_i^{(l)} \\
			\bar{\sigma}_o^{(l)} \bar{\sigma}_i^{(l)}
		\end{bmatrix}.
	\end{aligned}
\end{equation*}
Next, consider the whole state vector $\boldsymbol{x}_k^{(l)} = \begin{bmatrix}
	\boldsymbol{c}_{k}^{(l)^\top} & \boldsymbol{h}_{k}^{(l)^\top}
\end{bmatrix}^\top$. We apply the infinity norm on both sides of the inequality in \eqref{eq:state_update_matrix} and iterate back to $k = 0$, obtaining:
\begin{subequations}
	\label{eq:state_bound_from_0}
	\begin{align}
		\label{eq:state_bound_from_0_beta}
		\normsmall{\boldsymbol{x}_{k}^{(l)}}{\infty}
		& \leq \norm{\left( \mathbf{A}^{\!(l)}\right)^{k}}{\infty} \normsmall{\boldsymbol{x}_{0}^{(l)}}{\infty} +                                                                \\
		\label{eq:state_bound_from_0_gamma_u}
		& \ + \norm{\sum_{h = 0}^{k - 1} \left( \mathbf{A}^{\!(l)}\right)^{\!k\!-\!1\!-\!h\!} \mathbf{B}_u^{(l)} \normsmall{\tilde{\boldsymbol{u}}_{h}^{(l)}}{\infty}}{\infty} + \\
		\label{eq:state_bound_from_0_gamma_b}
		& \ + \norm{\sum_{h = 0}^{k - 1} \left( \mathbf{A}^{\!(l)}\right)^{\!k\!-\!1\!-\!h\!} \mathbf{B}_b^{(l)} \normsmall{\boldsymbol{b}_g^{(l)}}{\infty}}{\infty}.
	\end{align}
\end{subequations}
Now, starting from \eqref{eq:state_bound_from_0} and assuming that the condition in \eqref{eq:lstm_stability_condition} holds, we derive the $\beta \in \mathcal{KL}$ and $\gamma_u, \gamma_b \in \mathcal{K}_\infty$ functions in Definition \ref{def:strict_iss} to prove the \strictISS{} of the $l$-th \LSTM{} layer. Notice that, since $\bar{\sigma}_o^{(l)} \in (0, 1)$ due to \eqref{eq:sigma_tanh_bound}, we have:
\begin{equation}
	\label{eq:infty_norm_A}
	\norm{\mathbf{A}^{\!(l)}}{\infty} = \bar{\sigma}_f^{(l)} + \bar{\sigma}_i^{(l)} \normsmall{R_g^{(l)}}{\infty}.
\end{equation}
The relationship in \eqref{eq:infty_norm_A} only applies when taking the infinity norm of $\mathbf{A}^{\!(l)}$ and allows us to derive the condition in \eqref{eq:lstm_stability_condition}, which depends on fewer parameters compared to \cite[Proposition 1]{terzi_learning_2021} (see also Remark \ref{rem:diff_from_original_work}). Instead, when analyzing the \ISSp{2} of the $l$-th \LSTM{} layer, we have to consider the spectral radius of $\mathbf{A}^{\!(l)}$ rather than its infinity norm in \eqref{eq:infty_norm_A}. Consequently, \eqref{eq:infty_norm_A} is one of the main differences from \cite[Proof of Theorem 1]{terzi_learning_2021} and allows us to compute explicitly $\beta \in \mathcal{KL}$ and $\gamma_u, \gamma_b \in \mathcal{K}_\infty$ as well as to obtain the condition in \eqref{eq:lstm_stability_condition} instead of the one in \cite[Proposition 1]{terzi_learning_2021}. 

Moving on, applying \eqref{eq:infty_norm_A} to the power of $\mathbf{A}^{\!(l)}$ in \eqref{eq:state_bound_from_0_beta}:
\begin{equation}
	\label{eq:infty_norm_A_product}
	\norm{\left( \mathbf{A}^{\!(l)}\right)^{k}}{\infty}\!\!\!\leq\!\prod_{h=1}^{k} \norm{\mathbf{A}^{\!(l)}}{\infty} \!\!\! = \! \left(\bar{\sigma}_f^{(l)} + \bar{\sigma}_i^{(l)} \normsmall{R_g^{(l)}}{\infty}\right)^k\!.
\end{equation}
Then, function $\beta$ can be easily derived from the right side of \eqref{eq:state_bound_from_0_beta} and \eqref{eq:infty_norm_A_product}, leading to:
\begin{equation}
	\label{eq:function_beta_ISS}
	\beta(\normsmall{\boldsymbol{x}_0^{(l)}}{\infty}, k) = \left(\bar{\sigma}_f^{(l)} + \bar{\sigma}_i^{(l)} \normsmall{R_g^{(l)}}{\infty}\right)^k \normsmall{\boldsymbol{x}_{0}^{(l)}}{\infty},
\end{equation}
which is clearly a $\mathcal{KL}$ function \cite{jiang_input--state_2001} due to the condition in \eqref{eq:lstm_stability_condition}. Next, to obtain $\gamma_{u}$, we derive an upper bound on \eqref{eq:state_bound_from_0_gamma_u} by applying \eqref{eq:infty_norm_A_product} as follows:
\begin{align}
	\label{eq:computations_for_gamma_u}
	& \norm{\sum_{h = 0}^{k - 1} \left( \mathbf{A}^{\!(l)}\right)^{\!k\!-\!1\!-\!h\!} \mathbf{B}_u^{(l)} \normsmall{\tilde{\boldsymbol{u}}_{h}^{(l)}}{\infty}}{\infty} \leq \ldots                                                                                         \\
	& \quad \leq \normsmall{\mathbf{B}_u^{(l)}}{\infty} \sum_{h = 0}^{k - 1} \norm{\left( \mathbf{A}^{\!(l)}\right)^{\!k\!-\!1\!-\!h\!}}{\infty} \normsmall{\tilde{\boldsymbol{u}}_{h}^{(l)}}{\infty} \nonumber                                                            \\
	& \quad \leq \normsmall{\mathbf{B}_u^{(l)}}{\infty} \max_{0 \leq h < k} \normsmall{\tilde{\boldsymbol{u}}_{h}^{(l)}}{\infty} \sum_{h = 0}^{k - 1} \norm{\left( \mathbf{A}^{\!(l)}\right)^{\!k\!-\!1\!-\!h\!}}{\infty} \nonumber                                        \\
	& \quad \leq \normsmall{\mathbf{B}_u^{(l)}}{\infty} \max_{0 \leq h < k} \normsmall{\tilde{\boldsymbol{u}}_{h}^{(l)}}{\infty} \sum_{h = 0}^{k - 1} \left(\bar{\sigma}_f^{(l)} + \bar{\sigma}_i^{(l)} \normsmall{R_g^{(l)}}{\infty}\right)^{\!k\!-\!1\!-\!h\!} \nonumber
\end{align}
Notice that the last multiplicative term in the above equation is a geometric series that can be bounded as follows:
\begin{align*}
	\sum_{h = 0}^{k - 1} \left(\bar{\sigma}_f^{(l)} + \bar{\sigma}_i^{(l)} \normsmall{R_g^{(l)}}{\infty}\right)^{\!k\!-\!1\!-\!h\!} & = \frac{1 - \left(\bar{\sigma}_f^{(l)} + \bar{\sigma}_i^{(l)} \normsmall{R_g^{(l)}}{\infty}\right)^k}{1 - \left(\bar{\sigma}_f^{(l)} + \bar{\sigma}_i^{(l)} \normsmall{R_g^{(l)}}{\infty}\right)} \nonumber \\
	\label{eq:geometric_series_inequality}
	& \leq \frac{1}{1 - \left(\bar{\sigma}_f^{(l)} + \bar{\sigma}_i^{(l)} \normsmall{R_g^{(l)}}{\infty}\right)}
\end{align*}
since $0 < \bar{\sigma}_f^{(l)} + \bar{\sigma}_i^{(l)} \normsmall{R_g^{(l)}}{\infty} < 1$ due to \eqref{eq:sigma_tanh_bound} and \eqref{eq:lstm_stability_condition}. By substituting the previous result in \eqref{eq:computations_for_gamma_u}, we obtain the expression for $\gamma_u$:
\begin{equation}
	\label{eq:function_gamma_u_ISS}
	\gamma_u\left( \max_{0 \leq h < k} \normsmall{\tilde{\boldsymbol{u}}_{h}^{(l)}}{\infty} \right) = \frac{\normsmall{\mathbf{B}_u^{(l)}}{\infty} \max_{0 \leq h < k} \normsmall{\tilde{\boldsymbol{u}}_{h}^{(l)}}{\infty}}{1 - \left(\bar{\sigma}_f^{(l)} + \bar{\sigma}_i^{(l)} \normsmall{R_g^{(l)}}{\infty}\right)},
\end{equation}
which is a $\mathcal{K}_\infty$ function \cite{jiang_input--state_2001}. Finally, we can derive $\gamma_b \in \mathcal{K}_\infty$ in a fashion similar to \eqref{eq:function_gamma_u_ISS}, starting from \eqref{eq:state_bound_from_0_gamma_b} instead of \eqref{eq:state_bound_from_0_gamma_u}:
\begin{equation}
	\label{eq:function_gamma_b_ISS}
	\gamma_b \left( \normsmall{\boldsymbol{b}_g^{(l)}}{\infty} \right) = \frac{\normsmall{\mathbf{B}_b^{(l)}}{\infty}\normsmall{\boldsymbol{b}_g^{(l)}}{\infty}}{1 - \left(\bar{\sigma}_f^{(l)} + \bar{\sigma}_i^{(l)} \normsmall{R_g^{(l)}}{\infty}\right)}.
\end{equation}
In conclusion, according to Definition \ref{def:strict_iss}, under Assumption \ref{as:boundedness_of_input} and Assumption \ref{ass:initial_hidden_states}, if the condition in \eqref{eq:lstm_stability_condition} holds, the $l$-th \LSTM{} layer with state updates in \eqref{eq:LSTM_layer_cell_state_update}/\eqref{eq:LSTM_layer_hidden_state_update} is \strictISS{} since there exist $\beta \in \mathcal{KL}$ in \eqref{eq:function_beta_ISS}, $\gamma_u \in \mathcal{K}_\infty$ in \eqref{eq:function_gamma_u_ISS}, and $\gamma_b \in \mathcal{K}_\infty$ in \eqref{eq:function_gamma_b_ISS} such that
\begin{align*}
	\normsmall{\boldsymbol{x}_{k}^{(l)}}{\infty} & \leq \beta(\normsmall{\boldsymbol{x}_0^{(l)}}{\infty}, k) + \gamma_u\left( \max_{0 \leq h < k} \normsmall{\tilde{\boldsymbol{u}}_{h}^{(l)}}{\infty} \right) + \\
	& \quad + \gamma_b \left( \normsmall{\boldsymbol{b}_g^{(l)}}{\infty} \right)
\end{align*}
holds for any $\boldsymbol{x}_0^{(l)} \in \mathbb{R}^{n_{\mathrm{hu}}^{(l)}} \times (-1, 1)^{n_{\mathrm{hu}}^{(l)}}$, any input sequence $\tilde{\mathbf{u}}^{(l)} = \{\tilde{\boldsymbol{u}}_{h}^{(l)} \in \tilde{\mathcal{U}}^{(l)}: h \in \{0, \ldots, k-1\}\}$, any bias $\boldsymbol{b}_g^{(l)} \in \mathbb{R}^{n_{\mathrm{hu}}}$, and any $k \in \mathbb{N}$.
\hfill$\square$

\paragraphfont{Proof of Theorem \ref{th:cascaded_lstm_iss_fc}}
The network consists of $L$ subsystems. As shown by \authornote{Jiang}{jiang_input--state_2001}, a cascade of \strictISS{} systems remains \strictISS{}. Consequently, the complete \LSTM{} network is \strictISS{} if \eqref{eq:lstm_stability_condition} holds $\forall l \in \{1, \ldots, L\}$. 
\hfill$\square$
\end{document}